\documentclass[preprint,reqno]{elsarticle}
\usepackage{enumitem,graphicx}
\usepackage[usenames, dvipsnames]{color}
\usepackage{amsmath,amscd,amssymb,latexsym}
\usepackage{indentfirst}
\usepackage{pstricks}
\usepackage{setspace}
\usepackage{amsthm}
\usepackage{pifont}
\newcommand{\CO}[2]{ \left\langle #1 , #2 \right\rangle}
\DeclareMathSymbol{\R}{\mathbin}{AMSb}{"52}

\newtheorem{defi}{Definition}[section]
\newtheorem{thm}{Theorem}[section]

\newtheorem{rem}{Remark}[section]
\newtheorem{lem}{Lemma}[section]
\newtheorem{pro}{Proposition}[section]

\journal{Journal of Differential Equations}
\begin{document}
\begin{frontmatter}
\title{Super critical problems with concave and convex nonlinearities in $\mathbb R^N$}
\author[JMD]{J. M. do \'O\corref{mycorrespondingauthor}}
\ead{jmbo@pq.cnpq.br}

\cortext[mycorrespondingauthor]{Corresponding author}
\address[JMD]{Department of Mathematics, Bras\'{\i}lia University, 70910-900, Bras\'ilia, DF, Brazil.}

\author[Pawan Kumar Mishra]{P. K. Mishra}
\ead{pawanmishra31284@gmail.com}

\address[Pawan Kumar Mishra]{Department of Mathematics, Federal University of Para\'{\i}ba, 58051-900, Jo\~ao Pessoa, PB, Brazil.}

\author[A. Moameni]{A. Moameni}
\ead{momeni@math.carleton.ca}

\address[A. Moameni]{School of Mathematics and Statistics, Carleton University, Ottawa, Ontario, Canada.}
\begin{abstract}
In this paper, by utilizing a newly established  variational principle  on convex sets, we provide  an existence and multiplicity result for a class of semilinear elliptic problems defined on the  whole $\mathbb R^N$  with nonlinearities involving sublinear and superlinear terms. We shall impose no growth restriction on the nonlinear term and consequently our problem can be super-critical by means of Sobolev spaces.
\end{abstract}
\end{frontmatter}
\date{\today}

\section{Introduction}
In this paper, we aim to prove a  multiplicity result for the the class of  super linear problems of the form,
\begin{equation}\label{P}\tag{$P_\lambda$}
\left\{\begin{array}{lll}
&-\Delta u+V(x)u=f(u)+\lambda |u|^{q-2}u \;\;x\in \mathbb R^N,\\
&u\in H^1(\mathbb R^N), \displaystyle \int_{\mathbb R^N}V(x)|u|^2\,\mathrm{d}x<\infty,
\end{array}\right.
\end{equation}
 where $N\geq 2$, $1<q<2$ and $\lambda>0$ is a real parameter. 
We study the above problem for  the following two cases:\\
(1) $N\geq 3$ and $f(u)=|u|^{p-2}u$ for $p>2$.\\
(2) $N=2$ with the the following two assumptions on the function $f,$

$(f1)$ $f:\mathbb R \to \mathbb R$ is an  odd continuous function with $f(t)\geq 0$ for $t\geq 0$ and $f(0)=0;$ 

$(f2)$ There exists $\nu>1$ such that $\displaystyle \lim_{t\rightarrow 0}\frac{f(t)}{t^{\nu}}=0$.\\
We shall also impose the following conditions on the potential $V(x)$,\\
 $\textbf (V1)$ The function $V : \mathbb R^N \rightarrow \mathbb R $ is continuous and
$0<V_0\leq V(x)$ for all $x \in \mathbb R^N$;\\
$\textbf(V2)$ The function $1/V\in L^1(\mathbb R^N)$.

After the eminent work of Ambrosetti-Brezis-Cerami \cite{MR1276168}, the class of problems under consideration has been studied comprehensivly in bounded domains, see  \cite{MR1608057,MR1301008, MR1712564,MR1971261,MR1721723,  MR1983694} and  references therein. Using sub and super solution method, authors in  \cite{MR1276168} have proved the  existence of two positive solutions with the nonlinearity $f_\lambda(x, u)=\lambda u^q+u^p$ satisfying $0<q<1<p$. The gravity of results lies in the fact that there was no control on $p$ from above. Along with many results,  when $p\leq (N+2)/(N-2)$, the existence of infinitly many solution was also establised for suitable choice of parameter $\lambda$. Besides \cite{MR1276168}, we refer interested readers to see \cite{MR1396658, Abbas} also for concave and convex problems  with super critical growth in bounded domains for the existence of infinitly many solutions, where authors have adopted different techniques. In \cite{MR1396658} authors have applied a trucation argument while in \cite{Abbas} authors have adopted a new abstract variational principle discussed in \cite{Mo6} (see also \cite{Mo5}).

 In case of  $\mathbb R^N$, lesser has been explored for the elliptic problem involving concave and convex growth, see \cite{ MR1491844, MR1879863,MR1491612,   MR2185298, MR2557956} with no claim of citing all of them. To begin with, authors in \cite{MR1879863} have attempted to give existence results based on the method of successive approximations with no restrictions on the growth of super linear term. With a subcritical control over superlinear term, authors in \cite{MR2185298} have proved the existence of infinitely many nodal solutions for Schr\"odinger equation with concave-convex nonlinearity. A similar class of problem with sign changing weights has been studied in \cite{MR2557956} for the existence of multiple positive solutions, using the idea of Nehari manifold.

As far as  super critical concave and convex problem on whole $\mathbb R^N$ are concerned, we are only aware of the work in
 \cite{MR1879863} in which the existence of a single solution has been proved. In this paper we shall prove both existence and multiplicity. To be precise, we prove the following results 
 related with the problem \eqref{P}. 
\begin{thm}\label{thm1}
Assume that $1 < q < 2 < p$ and $N\geq 3$. Then  there exists $\Lambda_0 > 0$ such that for each
$\lambda\in (0,\Lambda_0)$ problem \eqref{P} has at least one  positive  solution with a  negative energy.
\end{thm}
The next theorem is about the multiplicity result for the super-critical case.
\begin{thm}\label{thm2}
Assume that $1 < q < 2 < p$ and $N\geq 3$. Then there exists $\Lambda_0 > 0$ such that for each $\lambda \in (0,\Lambda_0)$ problem \eqref{P} has infinitely many distinct nontrivial solutions with negative energy.
 \end{thm}
 The above results can be considered as an extension of results in \cite{MR2185298} in critical and super critical case. 
  Now we state the following result in reference to  the two dimensional case.
  \begin{thm}\label{thm3}
Assume that $1 < q < 2$. Then  there exists $\Lambda_{1} > 0$ such that for each
$\lambda\in (0,\Lambda_1)$ problem \eqref{P} has at least one positive  solution with a negative energy.
\end{thm}
\begin{rem} We remark that the assumptions $(V1)$ and $(V2)$ do not imply that $V(x)$ is coercive. For example,
 $V (x_1, x_2,...x_N)= 1+x_1^2[\sin^2(2\pi x_1)+x_2^2+x_3^2+....x_N^2]^\alpha$ for $\alpha>N$ satisfies $(V1)-(V2)$ but it is not coercive. 
 \end{rem}
\begin{rem}
A typical example satisfying $(f1)-(f2)$ can be 
$$
f(t)=t^{2\alpha+1}\exp({\beta t^2}), \; \alpha\in \mathbb N\;\; \textrm{such that}\; 2\alpha+1>\nu\;\;\textrm{and}\;\beta\in \mathbb R.
$$ 
\end{rem}
\begin{rem}
As one of the applications of the above problem, we remark that the concave-convex problems arise in the study of anisotropic continuous media. We refer readers to see the introduction in \cite{RAD} for several applications of this kind of problems.
\end{rem}
To prove these results we follow an  idea based on  variational principles on convex sets. One difficulty while dealing with  problems in unbounded domains is to choose a suitable convex set which has a required tolerance with the appropriate solution space so that one can  apply the abstract result   established recently in \cite{Mo6}.

The outline of the paper is as follows. In Section 2,  we shall recall  a new variational principle established in \cite{Mo5, Mo6} that paves a way to do critical point theory on convex sets and yet to obtain critical points with respect to the whole space. In section 3, we give some preliminary results  required for our variational setup. In section 4, we prove the existence result in  Theorem \ref{thm1} while Section 5 is devoted to the existence of infinity many solutions and the proof of Theorem \ref{thm2}.  Finally, we conclude this paper by dealing with the two dimensional case and the proof of Theorem \ref{thm3} in Section 6.
 
 \section{A variational principle}
Let $V$ be a reflexive Banach space, $V^*$ its topological dual
and $K$ be a convex and weakly closed subset of $V$.
Assume that $\Psi : V \rightarrow \mathbb{R} \cup \{+\infty\}$
is a proper, convex, lower semi-continuous function which is
G\^ateaux differentiable on $K$. The G\^ateaux derivative of
$\Psi$ at each point $u \in K$ will be denoted by $D\Psi(u)$.
The restriction of $\Psi$ to $K$ is denoted by $\Psi_K$ and
defined by
\begin{eqnarray}
\Psi_K(u)=\left\{
  \begin{array}{ll}
      \Psi(u), & u \in K, \\
    +\infty, & u \not \in K.
  \end{array}
\right.
\end{eqnarray}
For a given functional $\Phi \in C^{1}(V, \mathbb{R})$ denoted
by $\Phi'(u) \in V^{*}$ its derivative and consider the
functional $\mathcal I_K: V \to (-\infty, +\infty]$ defined by
 \begin{eqnarray*}
 \mathcal I_K(u):= \Psi_K(u)-\Phi(u).
 \end{eqnarray*}
 \\
According to Szulkin \cite{MR837231}, we have the following
definition for critical points of $\mathcal I_K$.

\begin{defi}\label{cpd}
A point $u_{0}\in V$ is said to be a critical point of $\mathcal I_K$ if
$\mathcal I_K(u_{0}) \in \mathbb{R}$ and if it satisfies the following
inequality
\begin{equation}\label{cpt}
\CO{\Phi'(u_{0})}{ u_{0}-v} + \Psi_K(v)- \Psi_K(u_{0}) \geq 0,
\qquad \forall v\in V,
\end{equation}
where $\langle., .\rangle$ is the duality pairing between $V$
and its dual $V^{*}$.
\end{defi}

\begin{pro}\label{minimum}
Each local minimum of $\mathcal I_K$ is necessarily a critical point of $\mathcal I_K$.
\end{pro}
\begin{proof} Let $u$ be a local minimum of $\mathcal I_K$. Using convexity of $\Psi_K$, it follows that for all small $t > 0$,
\begin{align*}
0 \leq \mathcal I_K\left((1 -t)u + tv\right) - \mathcal I_K(u) &= \Phi\left(u + t(v -u)\right) - \Phi(u) + \Psi_K\left((1 -t)u + tv\right) - \Psi_K(u)\\
&\leq \Phi\left(u + t(v -u)\right) - \Phi(u) + t\left(\Psi_K(v) - \Psi_K(u)\right).
\end{align*}
Dividing by $t$ and letting $t\rightarrow 0^+$ we obtain \eqref{cpt}.
\end{proof}

We also recall the notion of point-wise invariance condition from \cite{Mo6}.
\begin{defi}\label{def-invar}
We say that the triple $(\Psi, K, \Phi)$ satisfies the point-wise invariance condition at a 
 point $u_{0}\in V$ if there exists a convex G\^ateaux differentiable function $G: V\to \R$ and a point $v_0 \in K$ such that 
 \[D\Psi(v_0)+DG(v_0)=\Phi'(u_0)+DG(u_0).\]
\end{defi}
We shall now recall the following variational principle
established recently in \cite{Mo6} (see also \cite{Mo5}).
\begin{thm}\label{con2}
Let $V$ be a reflexive Banach space and $K$ be a convex and
weakly closed subset of $V$. Let $\Psi : V \rightarrow
\mathbb{R}\cup \{+\infty\}$ be a convex, lower semi-continuous
function which is G\^ateaux differentiable on $K$ and let $\Phi
\in C^{1}(V, \mathbb{R})$. Assume that  the following two assertions hold:\begin{enumerate}
\item[$(i)$] The functional $\mathcal I_K: V \rightarrow \mathbb{R}
\cup \{+\infty\}$ defined by $\mathcal I_K(u):= \Psi_K(u)-\Phi(u)$ has a
critical point $u_{0}\in V$ as in Definition \ref{cpd} and;
\\
\item[$(ii)$] the triple $(\Psi, K, \Phi)$ satisfies the point-wise invariance condition at the  
 point $u_{0}$.
\end{enumerate} 
Then $u_{0}\in K$ is a solution of the equation 
 \begin{equation} \label{equ1} D \Psi(u) =\Phi'(u).
 \end{equation}
 \end{thm}
 
 We shall now adapt the latter theorem to our case. Consider the Banach space $\mathcal V = E_V \cap L^p(\mathbb R^N)$  equipped with the following norm
\begin{align*}
\| u \| := \| u \|_{E_V} + \| u \|_{L^p(\mathbb R^N)},
\end{align*}
where
\[
E_V= \left\{u\in H^1(\mathbb R^N): \displaystyle \int_{\mathbb R^N}V(x)|u|^2\,\mathrm{d}x<\infty \right\}
\]
and
\[
\|u\|_{E_V}=\left(\displaystyle \int_{\mathbb R^N}|\nabla u|^2 \,\mathrm{d}x + \displaystyle\int_{\mathbb R^N}V(x)|u|^2 \,\mathrm{d}x \right)^\frac12.
\]
 Let  $\mathcal I : \mathcal V\rightarrow \mathbb{R}$ be the Euler-Lagrange functional related to \eqref{P}, given as
\begin{equation*}
\mathcal I(u) = \frac{1}{2} \int_{\mathbb R^N} | \nabla u | ^{2} \,\mathrm{d}x+\frac{1}{2} \int_{\mathbb R^N}V(x)|u|^2\,\mathrm{d}x - \frac{1}{p} \int_{\mathbb R^N}| u | ^{p} \,\mathrm{d}x - \frac{\lambda}{q} \int_{\mathbb R^N} | u | ^{q} \,\mathrm{d}x .
\end{equation*}

Define  the function $\Phi : \mathcal V\rightarrow \mathbb{R}$ by
\begin{align*}
\Phi (u) = \frac{1}{p} \int_{\mathbb R^N}| u | ^{p} \,\mathrm{d}x+ \frac{\lambda}{q}\int_{\mathbb R^N} | u | ^{q} \,\mathrm{d}x ,
\end{align*}
Note that $\Phi \in C^1(\mathcal V;\mathbb R).$
 Define
 $\Psi : \mathcal V \rightarrow \mathbb R$ by
\begin{equation*}
\Psi(u) = \frac{1}{2} \int_{\mathbb R^N} | \nabla u | ^{2} \,\mathrm{d}x+\frac12 \int_{\mathbb R^N} V(x)|u|^2\,\mathrm{d}x.
\end{equation*}
Let $K$ be a convex and weakly closed subset of $\mathcal V$. Then the restriction of $\Psi$ over $K$  is denoted by $\Psi_{K}$ and defined as
\begin{eqnarray}\label{psitru}
\Psi_K(u)=\left\{
  \begin{array}{ll}
      \Psi(u), & u \in K, \\
    +\infty, & u \not \in K.
  \end{array}
\right.
\end{eqnarray}
Finally, let us introduce  the functional $\mathcal I_K: \mathcal V \to (-\infty, +\infty]$ defined by
 \begin{eqnarray}\label{resfun0}
 \mathcal I_K(u):= \Psi_K(u)-\Phi(u).
 \end{eqnarray}

For the convenience of the reader, we  shall prove a simplified version of Theorem \ref{con2}, suitable to our problem \eqref{P}.

 \begin{thm}\label{12v2}
Let $\mathcal V=E_V \cap L^p(\mathbb R^N)$ as defined before, and let  and $K$ be a convex and weakly closed subset of $\mathcal V$. If the following two assertions hold:
\begin{enumerate}
\item[$(i)$] The functional $\mathcal I_K: \mathcal V \rightarrow \mathbb{R} \cup \{+\infty\}$ defined in \eqref{resfun0}  has a critical point $\overline u\in \mathcal V$ as in Definition \ref{cpd}, and;

\item[ $(ii)$] there exists $\overline v\in K$ such that $-\Delta \overline v +V(x) \overline v= D \Phi(\overline u)=\overline u |\overline  u|^{p-2}+\lambda \overline u | \overline u| ^{q - 2}$ in the weak sense.
\end{enumerate}
Then $\overline u\in K$ is a solution of the equation
 \begin{equation} \label{equ1}
 -\Delta  u+V(x) u = u |  u|^{p-2}+\lambda u | u| ^{q - 2}.
  \end{equation}
 \end{thm}
\textbf{Proof.} Since $\overline u$ is a critical point of $\mathcal I_K(u)=\Psi_K(u)-\Phi(u),$
it follows from Definition \ref{cpd} that
\begin{equation}\label{ineq0}
\Psi_K(v)-\Psi_K(\overline u)\geq \langle D \Phi(\overline u),v-\overline u\rangle,\quad \forall v\in V,
\end{equation}
where $\langle D \Phi(\overline u),v-\overline u\rangle=\int_{\mathbb R^N} D \Phi(\overline u)(v-\overline u)\, \,\mathrm{d}x.$ Which leads to
\begin{equation}\label{thmi}
\frac{1}{2} \int_{\mathbb R^N} (| \nabla v | ^{2}-| \nabla \overline u | ^{2}) \,\mathrm{d}x+\frac{1}{2}\int_{\mathbb R^N} V(x)(|v|^2-|\overline u|^2)\,\mathrm{d}x\geq \langle D \Phi(\overline u),v-\overline u\rangle ,\quad \forall v\in V.
\end{equation}
It follows from the second assumption in the theorem that there exists $\overline v\in K$ such that
\begin{equation}\label{thmii}
 \int_{\mathbb R^N}\nabla \overline v. \nabla \eta \,\mathrm{d}x + \int_{\mathbb R^N} V(x)\overline v \eta  \,\mathrm{d}x = \int_{\mathbb R^N} D\Phi(\overline u)\eta \,\mathrm{d}x, \quad \forall \eta \in V..
\end{equation}
Now putting $\eta=\overline u-\overline v$ in \eqref{thmii}, we get
\begin{equation}\label{iicons}
 \int_{\mathbb R^N}\nabla \overline v. \nabla(\overline u-\overline v)  \,\mathrm{d}x+ \int_{\mathbb R^N} V(x)\overline v (\overline u-\overline v) \,\mathrm{d}x = \int_{\mathbb R^N} D\Phi(\overline u)(\overline u-\overline v) \,\mathrm{d}x.
\end{equation}
Now substituting $v=\overline v$ in \eqref{thmi} and using \eqref{iicons}, we get
\begin{equation}\label{mixinii}
\frac{1}{2} \int_{\mathbb R^N} (| \nabla \overline v | ^{2}-| \nabla \overline u | ^{2}) \,\mathrm{d}x+\frac{1}{2}\int_{\mathbb R^N} V(x)(|\overline v|^2-|\overline u|^2) \,\mathrm{d}x\geq \int_{\mathbb R^N}\nabla \overline v. \nabla(\overline u-\overline v)  \,\mathrm{d}x+ \int_{\mathbb R^N} V(x)\overline v (\overline u-\overline v) \,\mathrm{d}x.
\end{equation}
On the other hand, by the convexity of $\Psi$, we get
\begin{equation}\label{lst}
\frac{1}{2} \int_{\mathbb R^N} (| \nabla \overline u | ^{2}-| \nabla \overline v | ^{2}) \,\mathrm{d}x+\frac{1}{2}\int_{\mathbb R^N} V(x)(|\overline u|^2-|\overline v|^2)\,\mathrm{d}x\geq \int_{\mathbb R^N}\nabla \overline v. \nabla(\overline v-\overline u)  \,\mathrm{d}x+ \int_{\mathbb R^N} V(x)\overline v (\overline v-\overline u) \,\mathrm{d}x.
\end{equation}
Combining \eqref{mixinii} and \eqref{lst}, we get
\[
\frac{1}{2} \int_{\mathbb R^N} (| \nabla \overline v | ^{2}-| \nabla \overline u | ^{2}) \,\mathrm{d}x+\frac{1}{2}\int_{\mathbb R^N} V(x)(|\overline v|^2-|\overline u|^2)\,\mathrm{d}x= \int_{\mathbb R^N}\nabla \overline v. \nabla(\overline u-\overline v)  \,\mathrm{d}x+ \int_{\mathbb R^N} V(x)\overline v (\overline u-\overline v) \,\mathrm{d}x.
\]
Indeed the last equation is equivalent to
\[
\frac{1}{2} \int_{\mathbb R^N} | \nabla \overline v-\nabla \overline u | ^{2} \,\mathrm{d}x+\frac{1}{2}\int_{\mathbb R^N} V(x)|\overline v-\overline u|^2\,\mathrm{d}x=0.
\]
Since $V(x)>0$, we get $\overline v=\overline u$. Using this observation in \eqref{thmii}, we get the required result . This completes the proof. \hfill $\square$\\

We remark that the condition $ii)$ in Theorem \ref{12v2} indeed shows that the triple $(\Psi, K, \Phi)$ satisfies the point-wise invariance condition  at $u_0$ given in Definition \ref{def-invar} . In fact, it corresponds to the case where $G=0$.  This is why Theorem \ref{12v2} is a very particular case of the general Theorem \ref{con2}. 
 \section{Preliminary results}
 In this section we prove some preliminary results required throughout the paper.  
We have the following result  for the compact inclusion of the space $E_V$ into suitable Lebesgue spaces.
\begin{lem}\label{embd}
Under the assumption $(V1)-(V2)$ and $N\geq 2$ the embedding $E_V\hookrightarrow L^\beta(\mathbb R^N)$ is compact for $\beta \in [1, 2^*)$ where $2^*=2N/(N-2)$ for $N>2$ and $2^*=\infty$ for $N=2.$
\end{lem}
\begin{proof}
By $(V1)$ the embedding $E_V\hookrightarrow H^1(\mathbb R^N)$ is continuous. Thus, $E_V\hookrightarrow L^\beta(\mathbb R^N)$  is
continuous for $\beta \in [2, 2^*)$. Moreover, if $u\in E_V$, we have
\[
\displaystyle \int_{\mathbb R^N} |u|\,\mathrm{d}x \leq \left(\displaystyle \int_{\mathbb R^N} (V(x))^{-1}\,\mathrm{d}x \right)^\frac12\|u\|_{E_V}.
\]
Therefore, by interpolation $E_V\hookrightarrow L^\beta(\mathbb R^N)$  is continuous for $\beta \in [1, 2^*)$. Now, let $\{u_n\}$ be
a bounded sequence in $E_V$, i.e. $\|u_n\|_{E_V}\leq C$ for some $C>0$. Hence, up to a subsequence, $u_n\rightharpoonup u_0$ weakly in $E_V$. Given $\epsilon > 0$,
we consider $R > 0$ such that
\[
\displaystyle \int_{|x|>R}(V(x))^{-1}\,\mathrm{d}x\leq \left[\frac{\epsilon}{2(C+\|u_0\|_{E_V})}\right]^2.
\]
This implies that
\[
\displaystyle \int_{|x|>R}|u_n-u_0|\,\mathrm{d}x\leq \displaystyle \left(\int_{|x|>R}(V(x))^{-1}\,\mathrm{d}x\right)^\frac12\|u_n-u_0\|_{E_V}
\]
and since $E_V\hookrightarrow L^1(B_R)$ is compact, if follows that there exists $n_0$ such that for all
$n>n_0$,
\[
 \displaystyle\int_{B_R}|u_n-u_0|\,\mathrm{d}x\leq \frac{\epsilon}{2}.
\]
Thus, $u_n\rightarrow u_0$ in $L^1(\mathbb R^N)$. Next, if $\beta \in [1, 2^*)$ then choose $\beta< \beta_0<2^*$  and use interpolation inequality, for
some $0 < \alpha \leq 1 $ to  get
\[
\|u_n-u_0\|_{L^\beta(\mathbb R^N)}\leq \|u_n-u_0\|_{L^1(\mathbb R^N)}^\alpha\|u_n-u_0\|_{L^{\beta_0}(\mathbb R^N)}^{1-\alpha} \rightarrow 0
\]
and the proof is complete.
\end{proof}
\begin{lem}\label{UEST}
Let  $(V1)-(V2)$ be satisfied and  let $g \in L^{\infty}(\mathbb R^N)$ for $N \geq 2.$  If $u \in E_V$ is a weak solution of the problem,
\begin{equation}\label{linear}-\Delta u+V(x)u=g(x),\end{equation}
then  $V_0 \|u\|_{L^{\infty}(\mathbb R^N)}\leq \|g\|_{L^{\infty}(\mathbb R^N)}.$ 
\end{lem}
\begin{proof} Take $h \in C_c^\infty(\mathbb R^N)$ and assume that $v \in E_V$ is a weak solution of $-\Delta v+V(x)v=h(x).$ We  show that $V_0 \|v\|_{L^1{(\mathbb R^N)}}\leq  \|h\|_{L^1(\mathbb R^N)}$. Let $\eta \in C^1(\mathbb R,\mathbb R)$ be such that $\eta(0)=0,$ $\eta' \geq 0$, $|\eta| \leq 1,$ and $\eta' \in L^{\infty}(\mathbb R).$  It can be easily deduced that $\eta(v) \in E_V.$ Since $v$ is a weak solution of $-\Delta v+V(x)v=h(x)$, it follows that
\[\int_{\mathbb R^N}\eta'(v)|\nabla v|^2\, \,\mathrm{d}x+\int_{\mathbb R^N} V(x)v \eta (v)\, \,\mathrm{d}x=\int_{\mathbb R^N} h(x) \eta (v) \,\mathrm{d}x\]
and therefore,
\begin{equation}\label{linear1}V_0 \int_{\mathbb R^N} v \eta (v)\,\mathrm{d}x\leq \int_{\mathbb R^N} h(x) \eta (v)\,\mathrm{d}x\leq \|h\|_{L^1(\mathbb R^N)}.\end{equation}

Given $\epsilon >0,$ let $\eta (v)=v/\sqrt{\epsilon +v^2}.$  It follows from (\ref{linear1}) that
\[V_0 \int_{\mathbb R^N} \frac {v^2}{\sqrt{\epsilon +v^2}} \,\mathrm{d}x \leq \|h\|_{L^1(\mathbb R^N)}.\]
Letting $\epsilon \to 0^+$ and applying Fatou's Lemma imply that $V_0 \|v\|_{L^1{\mathbb R^N}}\leq  \|h\|_{L^1(\mathbb R^N)}.$\\
On the other hand $u\in E_V$ is a weak solution of (\ref{linear}). Thus, 
\[\int_{\mathbb R^N} u h \,\mathrm{d}x=\int_{\mathbb R^N} \nabla v.\nabla u\,\mathrm{d}x+V(x)uv\,\mathrm{d}x=\int_{\mathbb R^N} gv \,\mathrm{d}x.\]
Therefore,
\[\Big |\int_{\mathbb R^N} u h \,\mathrm{d}x\Big | \leq \|v\|_{L^1(\mathbb R^N)}\|g\|_{L^{\infty}(\mathbb R^N)} \leq \frac{1}{V_0}\|h\|_{L^1(\mathbb R^N)}\|g\|_{L^{\infty}(\mathbb R^N)}.\] 
Since  $h \in C_c^\infty(\mathbb R^N)$ is arbitrary, we obtain  that  $V_0 \|u\|_{L^{\infty}(\mathbb R^N)}\leq \|g\|_{L^{\infty}(\mathbb R^N)}.$ 
\end{proof}
Recall that $\mathcal V = E_V \cap L^p(\mathbb R^N)$. 
To prove Theorem \ref{thm1}, keeping in mind the continuous inclusion of Lemma \ref{embd}, we  have the following construction of the closed set $K\subset E_V$,
 \begin{equation}\label{cwcK0}
 K=K(r) :=\big \{ u \in  \mathcal V \cap L^\infty(\mathbb R^N): \| u \| _{L^{\infty} (\mathbb R^N)} \leq r \big  \},
 \end{equation}
  for some $r > 0$ to be determined later. 

 In the next Lemma we show that the set $K$ is weakly closed.
 \begin{lem}
 Let $r>0$ be fixed then the set $K(r)$ defined in \eqref{cwcK0} is weakly closed in $\mathcal V.$
 \end{lem}
 \begin{proof}
 Take a sequence  $\{u_n\}$ in $K(r)$ such that $u_n\rightharpoonup u$ weakly in $\mathcal V$. Since $\mathcal V$ is reflexive, $u\in \mathcal V$. Now it remains to show that $u\in L^{\infty}(\mathbb R^N)$ and $\|u\|_{L^\infty(\mathbb R^N)}\leq r$. Since $u_n\rightharpoonup u $ in $\mathcal V$, it converges point wise up 
to a subsequence, i.e.,  $u_n(x) \rightarrow u(x)$ almost everywhere in $\mathbb R^N$. This implies that $|u(x)| =\lim_{n \to \infty }|u_n(x)|\leq r$ for a.e. $x \in \mathbb R^N$.  Thus, $\| u \| _{L^{\infty} (\mathbb R^N)} \leq r$. 
 \end{proof}

\section{Proof of Theorem \ref{thm1}}
We begin with the following elementary result which can be deduced by a straightforward  computation.
\begin{lem} \label{req1}
Let $1 < q < 2 < p$ and $V_0>0$ is same as defined in $(f1)$.  Then there exists  $\Lambda_0 > 0$ with the following properties.
\begin{enumerate}
\item For each $\lambda \in (0,\Lambda_0)$, there exist positive numbers $r_{1},  r_{2} \in \mathbb{R}$ with $r_{1} < r_{2}$ such that  $r \in  [r_{1}, r_{2}]$ if and only if   $ r^{p -1} + \lambda  r^{q -1}\leq V_0 r.$
\item  For $\lambda=\Lambda_0,$ there exists one and only one  $r>0$ such that $ r^{p -1} + \lambda  r^{q -1}= V_0 r.$
\item For $\lambda> \Lambda_0,$ there is no $r>0$ such that $ r^{p -1} + \lambda  r^{q -1}= V_0 r.$
\end{enumerate}
\end{lem}
The following Lemma is helpful in verifying the   condition $(ii)$ in Theorem \ref{12v2}.
\begin{lem} \label{3.2}
Assume that $1 < q < 2 < p$. Then
\begin{align*}
\Vert D\Phi (u) \Vert_{L^{\infty}(\mathbb R^N)} \leq r^{p -1} + \mu
r^{q -1}, \qquad \forall u \in K(r).
\end{align*} 
\end{lem}
\begin{proof}
By definition of $D\Phi (u)$ we have
\begin{align*}
\Vert D\Phi (u) \Vert_{L^{\infty}(\mathbb R^N)} &= \big \Vert u \vert
u\vert ^{p - 2} +\lambda u \vert u \vert ^{q - 2}
\big \Vert_{L^{\infty}(\mathbb R^N)} \\
& \leq \big \Vert u \vert u\vert ^{p - 2}\big \Vert_{L^{\infty}(\mathbb R^N)} +\lambda \big \Vert u \vert u \vert ^{q - 2} \big \Vert_{L^{\infty}(\mathbb R^N)}.
\end{align*} 
Therefore,
\begin{align*}
\Vert D\Phi (u) \Vert_{L^{\infty}(\mathbb R^N)} \leq \Vert u
\Vert_{L^{\infty}(\mathbb R^N)}^{p -1} + \lambda \Vert u
\Vert_{L^{\infty}(\mathbb R^N)}^{q -1}.
\end{align*} 
It follows from $u \in K(r)$ that
\begin{align*}
\Vert D\Phi (u) \Vert_{L^{\infty}(\mathbb R^N)} \leq r^{p -1} + \lambda
r^{q -1} ,
\end{align*} 
as desired.
\end{proof}

We are now in the position to state the following result addressing condition $(ii)$ in Theorem \ref{12v2}.
\begin{lem} \label{3.3}
Let  $1 < q < 2 < p$. Assume that $\Lambda_0> 0$ is given in Lemma \ref{req1} and  $\lambda \in (0, \Lambda_0).$ Let $r_1, r_2$ be given in part (1) of  Lemma \ref{req1}. Then   for each $r \in  [r_{1}, r_{2}]$  and   each $\overline{u}\in K(r)$ there exists $v \in K(r)$ such that
 \[-\Delta v+V(x)v=\overline{u} | \overline{u}| ^{p - 2} +\lambda \overline{u} | \overline{u} | ^{q - 2} .\]
\end{lem}

\begin{proof}
Let $g(x)=\overline{u} | \overline{u}| ^{p - 2} +\lambda \overline{u} | \overline{u} | ^{q - 2}.$
Since $\overline{u} \in K(r),$ it follows that $g \in L^\infty(\mathbb R^N).$ Since the embedding $E_V \hookrightarrow L^1(\mathbb R^N)$ is compact, the functional
\[Q(u)= \frac{1}{2} \int_{\mathbb R^N} | \nabla u | ^{2} \,\mathrm{d}x+\frac{1}{2} \int_{\mathbb R^N}V(x)|u|^2\,\mathrm{d}x-\int_{\mathbb R^N} g(x) u\,\mathrm{d}x,\]
is well-defined on $E_V$ and admits its minimum at some 
   $v \in E_V$ which indeed satisfies
\begin{align} \label{4m}
-\Delta v+V(x)v  = g(x)=D\Phi (\overline{u}),
\end{align}
in a weak sense. Since the right hand side is an element in $L^\infty(\mathbb R^N),$ it follows from Lemma \ref{UEST} that $V_0\|v\|_{L^\infty(\mathbb R^N)} \leq \|g\|_{L^\infty(\mathbb R^N)}$.
This together with  Lemma \ref{3.2} yield that
\begin{align*}
V_0\|v\|_{L^\infty(\mathbb R^N)}   \leq  r^{p -1} + \lambda  r^{q -1}.
\end{align*}
By Lemma \ref{req1},  for each $r \in  [r_{1}, r_{2}]$ we have that  $ r^{p -1} + \lambda  r^{q -1}\leq  V_0 r.$ Therefore,
\[\|v\|_{L^\infty(\mathbb R^N)} \leq r,\]
as desired.
\end{proof}

\textbf{Proof of Theorem \ref{thm1}.}
 Let $\Lambda_0$ be as in Lemma \ref{3.3} and $\lambda \in (0, \Lambda_0).$ Also, let  $r_1$ and $r_2$ be as in Lemma \ref{3.3} and define
\[K:=\big \{u \in K(r_2); \, \, \, u(x) \geq 0 \text{ a.e. } x \in \mathbb R^N\big \}.\]
 Now we proceed with the proof in the following  steps.\\
{\it Step 1.}  We show that there exists  $\overline{u} \in K$ such that $\mathcal I_K(\overline{u}) = \inf _{u \in E_V}\mathcal I_K(u)$. Then by Proposition \ref{minimum}, we conclude that  $\overline{u}$  is a critical point of $\mathcal I_K.$\\
Set $\eta := \inf _{u \in E_V}\mathcal I_K(u)$. So by definition of $\Psi_K$ for every $u \notin K$, we have $\mathcal I_K(u) = +\infty$ and therefore $\eta = \inf _{u \in K}\mathcal I_K(u)$. It follows that for every $u \in K$
\begin{align*}
\Phi (u) & = \dfrac{1}{p} \int_{\mathbb R^N} | u | ^{p} \,\mathrm{d}x+ \dfrac{\lambda}{q} \int_{\mathbb R^N} | u | ^{q} \,\mathrm{d}x \\
 &\leq \dfrac{r_2^{p-1}}{p} \int_{\mathbb R^N} | u |  \,\mathrm{d}x+ \dfrac{\lambda r_2^{q-1}}{q} \int_{\mathbb R^N} | u | \,\mathrm{d}x\\
&\leq  c_{1} \| u\|_{E_{V}} ,
\end{align*}
where we have used the embedding $E_{V} \hookrightarrow L^1(\mathbb R^N)$ due to Lemma \ref{embd}.  Thus, for $u \in K$ we have that 
 \begin{eqnarray}\label{coer}
  \mathcal I_K(u):= \Psi_K(u)-\Phi(u) \geq \frac{1}{2}\| u\|^2_{E_{V}} -c_{1} \| u\|_{E_{V}}, 
 \end{eqnarray}
from which we obtain that  $\eta> -\infty$.
Now, suppose that $\{u_{n}\}$ is a sequence in $K$ such that $\mathcal I_K(u_{n})\rightarrow \eta$. It follows from (\ref{coer}) and the definition of set $K$ that  the sequence $\{u_{n}\}$ is bounded in $E_{V}\cap L^\infty(\mathbb R^N)$. Using standard results in Sobolev spaces, after passing to a subsequence if necessary, there exists $\overline{u} \in E_V$ such that $u_n \rightharpoonup \overline{u} $ weakly in $E_{V}$. Moreover $u_n(x)\rightarrow \overline u(x)$ in $\mathbb R^N$ pointwise almost everywhere which implies $\overline u\in L^\infty(\mathbb R^N)$ with $\|\overline u\|_{L^\infty(\mathbb R^N)}\leq r_2$. As a consequence  $\overline u\in K$.  We now show that $\Phi(u_n) \to \Phi(\overline{u}).$  Indeed, we have that 
\[\dfrac{1}{p}  | u_n | ^{p} + \dfrac{\lambda}{q}  | u_n | ^{q}\leq \dfrac{r_2^{p-1}}{p}  | u_n |  + \dfrac{\lambda r_2^{q-1}}{q}  | u_n |.\]

Therefore, by the strong convergence $u_n \to \overline u$ in $L^1(\mathbb R^N)$ because of the compact embedding  $E_{V} \hookrightarrow L^1(\mathbb R^N)$ as in Lemma \ref{embd}, and the dominated convergence theorem we obtain that $\Phi(u_n) \to \Phi(\overline{u}).$

Therefore, $\mathcal I_K(\overline{u})\leq \liminf_{n \to \infty }\mathcal I_K(u_{n})$. So, $\mathcal I_K(\overline{u}) = \eta=\inf _{u \in \mathcal V}\mathcal I_K(u),$ and the proof of {\it Step 1} is complete.\\

{\it Step 2.}
 In this step we show that there exists $v\in K$ such that  $-\Delta v +V(x)v=\overline{u} | \overline{u}| ^{p - 2} +\lambda \overline{u} | \overline{u} | ^{q - 2} .$ By  Lemma \ref{3.3} together with the fact that $\overline{u} \in K(r_2)$ we obtain that $v\in K(r_2).$ To show that $v\in K,$ we shall need to verify that $v$ is non-negative almost every where.  But, this is a simple consequence
of the maximum principle and the fact  that $-\Delta v+V(x)v=\overline{u} | \overline{u}| ^{p - 2} +\lambda \overline{u} | \overline{u} | ^{q - 2} \geq 0.$ \\

It now follows from Theorem \ref{12v2} together with
{\it Step 1} and
{\it Step 2} that $\overline{u}$ is a solution of  (\ref{P}).  To complete the proof we shall show that $\overline{u}$ is non-trivial by proving that
 $\mathcal I_{K}(\overline{u}) = \inf_{u \in K} \mathcal I_{K}(u) < 0$.
Take $e \in K$. For $t\in [0,1],$ we have that  $te \in K$ and therefore
\begin{align*}
\mathcal I_{K}(te) & = \frac{1}{2} \int_{\mathbb R^N} | \nabla te | ^{2} \,\mathrm{d}x+\frac{1}{2} \int_{\mathbb R^N} V(x)|te | ^{2} \,\mathrm{d}x - \frac{1}{p} \int_{\mathbb R^N} | te | ^{p} \,\mathrm{d}x - \frac{\lambda}{q} \int_{\mathbb R^N} | te | ^{q} \,\mathrm{d}x
\\
& = t^{q}\left( \frac{ t^{2 - q}}{2} \int_{\mathbb R^N} | \nabla e | ^{2} \,\mathrm{d}x +\frac{t^{2-q}}{2} \int_{\mathbb R^N} V(x)|e | ^{2} \,\mathrm{d}x-\frac{ t^{p - q}}{p} \int_{\mathbb R^N} | e | ^{p} \,\mathrm{d}x - \frac{\lambda}{q} \int_{\mathbb R^N} | e | ^{q} \,\mathrm{d}x\right).
\end{align*}
Since $1 < q < 2 < p$, $\mathcal I_{K}(te)$ is negative for $t$ sufficiently small. Thus,  we can conclude that $\mathcal I_{K}(\overline{u}) < 0$. Thus, $\overline{u}$ is a non-trivial and non-negative solution of (\ref{P}). Finally,  it follows from the strong maximum principle that $\overline u>0$.
\hfill$\square$\\
\section{Proof of Theorem \ref{thm2}}
In this section, we prove Theorem \ref{thm2}.  According to \cite{MR837231}, say  that $\mathcal I_K$ satisfies the  compactness condition of Palais-Smale type provided,
\\
{(PS):}$\quad$ \emph{If $ \{u_{n}\}$ is a sequence such that $\mathcal I_K(u_{n})\rightarrow c \in \mathbb{R}$ and
\begin{align} \label{ps}
\langle D\Phi (u_{n}), u_{n} - v \rangle + \Psi_K (v) - \Psi_K(u_{n}) \geq -\epsilon_{n} \| v - u_{n}\|\qquad \forall v \in \mathcal V,
\end{align}
where $\epsilon_{n}\rightarrow 0$, then $\{u_{n}\}$  possesses a convergent subsequence}.

We recall  an important result  about critical points of even functions of the type $\mathcal I_K$. We shall begin with some preliminaries. Let $\Sigma$ be the of all symmetric subsets of $\mathcal V\setminus \{0\}$ which are closed in $\mathcal V$. A nonempty set $A \in \Sigma$ is said to have \emph{genus k} (denoted $\gamma(A) = k$) if $k$ is the smallest integer with the property that there exists an odd continuous mapping $h : A\rightarrow \mathbb{R}^{k} \setminus\{0\}$. If such an integer does not exist, $\gamma(A) = \infty$. For the empty set $\emptyset$ we define $\gamma(\emptyset) = 0$.
\begin{pro}\label{genus}
Let $A \in \Sigma$. If $A$ is a homeomorphic to $S^{k -1}$ by an odd homeomorphism, then $\gamma(A) = k$.
\end{pro}
Proof and a more detailed discussion of the notion of genus can be found in \cite{Rabinowitz 1} and \cite{Rabinowitz 2}.
\\
Let $\Theta$ be the collection of all nonempty closed and bounded subsets of $\mathcal V$. In $\Theta$ we introduce the Hausdorff metric distance (\cite{MR0090795}, $\S15, VII$), given by
\begin{align*}
\mathrm {dist}\;(A, B) =\max\left\{\sup_{a\in A}\mathrm{d}(a, B),\; \sup_{b\in B}\mathrm{d}(b, A)\right\}.
\end{align*}
The space $(\Theta, \mathrm{dist})$ is complete (\cite{MR0090795}, $\S29, IV$). Denote by $\Gamma$ the sub-collection of $\Theta$ consisting of all nonempty compact symmetric subsets of $\mathcal V$ and let
\begin{align}\label{Gam}
\Gamma_{j} = cl\{A\in \Gamma : 0\notin A, \gamma(A)\geq j \}
\end{align}
($cl$ is the closure in $\Gamma$). It is easy to verify that $\Gamma$ is closed in $\Theta$, so $(\Gamma, \mathrm{dist})$ and $(\Gamma_{j}, \mathrm{dist})$ are complete metric spaces.
The following Theorem is proved in \cite{MR837231}.
\begin{thm}\label{critical theorem} Let $\Psi : V \rightarrow
\mathbb{R}\cup \{+\infty\}$ be a convex, lower semi-continuous
function, and let $\Phi
\in C^{1}(V, \mathbb{R})$ and  define $\mathcal I =\Psi-\Phi.$
Suppose that $\mathcal I : \mathcal V\rightarrow (-\infty, +\infty]$ satisfies {(PS)}, $\mathcal I(0) = 0$ and $\Phi$, $\Psi$ are even. Define
\begin{align*}
c_{j} = \inf_{A\in \Gamma_{j}} \sup_{u \in A} \mathcal I(u).
\end{align*}
If $-\infty < c_{j} < 0$ for $j = 1, ... ,k$, then $\mathcal I$ has at least $k$ distinct pairs of nontrivial critical points by   means of Definition \ref{cpd}.
\end{thm}
Now, to prove Theorem \ref{thm2} we have the following construction of the closed set $K$.
 \begin{equation}\label{cwcK}
 K :=\big \{ u \in \mathcal V\cap L^{\infty}(\mathbb{R}^N) : \| u \| _{L^\infty(\mathbb R^N)} \leq r_2 \big  \},
 \end{equation}
 where $r_2$ is given in Lemma \ref{3.2}.
Before proving Theorem \ref{critical theorem} for our problem, we use the following result.
\begin{lem}\label{selfadjoint}
Assume the potential $V$ satisfies $(V1)-(V2)$. Then the Schr\"odinger operator, $-\Delta+V$, is self-adjoint.
\end{lem}
\textbf{Proof of Theorem \ref{thm2}.}
Let $\Lambda_0$ be as in Lemma \ref{3.2} and $\lambda \in (0, \Lambda_0).$  We first show  that the functional $\mathcal I_K$ has infinitely many distinct critical points by verifying Theorem \ref{critical theorem} in our set-up. It is obvious that the function
$\Phi$ is even and continuously differentiable. Also $\Psi_K$ is a proper, convex and lower semi-continuous even function. We now  verify ${(PS)}$. Let $\{u_n\}$ be a Palais-Smale sequence for $\mathcal I_K$ in $\mathcal V$ such that $\mathcal I_K(u_n)\rightarrow c$  for some $c\in \mathbb{R}$
and
\begin{align} \label{ps00}
\langle D\Phi (u_{n}), u_{n} - v \rangle + \Psi_K (v) -
\Psi_K(u_{n}) \geq -\epsilon_{n} \Vert v - u_{n}\Vert_{\mathcal V},\qquad
\forall v \in V,
\end{align}
where $\epsilon_{n}\rightarrow 0$.
 Since $c\in \mathbb R$, $\{u_n\}$ is bounded in $L^\infty(\mathbb R^N)$ ($u_n$ must belong to $K$ otherwise $\Psi_K(u_n)=\infty$ which contradicts that $c\in \mathbb R$). Moreover, it is easy to conclude that $\{u_n\}$ is bounded in $E_V$. Now using $\mathcal I_K(u_{n})\rightarrow c$, $\| u_n \|_{L^\infty(\mathbb R^N)} \leq r_2$ and the compact embedding   $E_V\hookrightarrow L^q(\mathbb R^N)$, it follows that $\{u_n\}$ is bounded in $L^p(\mathbb R^N)$ as well. Now since $\mathcal V$ is reflexive , there exists $\overline u\in \mathcal V$ such that $u_n\rightharpoonup \overline u$ in $\mathcal V$. Moreover following the previous argument as in Theorem \ref{thm1}, $\overline u\in L^{\infty}(\mathbb R^N)$ with $\|\overline u\|_{L^\infty(\mathbb R^N)}\leq r_2$. Consequently, it implies that $\overline u \in K$. Now we prove that $u_n \rightarrow \overline u$ strongly  in $\mathcal V$.
Following  the similar idea as in proof of Theorem \ref{thm1} combined with Lebesgue dominated convergence theorem and compact embedding of $E_V\to L^1(\mathbb R^N)$, we obtain that  
\begin{eqnarray} \label{late}\Phi(u_{n})\rightarrow \Phi(\overline{u}),\text{  and  } \langle D\Phi (u_{n}), u_{n} - v \rangle \to 0,\quad \forall v \in \mathcal V.\end{eqnarray}  As a result, invoking Brezis-Lieb Lemma, $u_n\rightarrow \overline u$ in $L^p(\mathbb R^N)$. 
 Moreover by \eqref{ps00}
\begin{align*}
\Psi(u_{n}) - \Psi(\overline{u}) - \langle D\Phi (u_{n}), u_{n}
- \overline{u}\rangle &\leq \epsilon_{n} \Vert \overline{u} -
u_{n}\Vert_{\mathcal V},
\end{align*}
and by  the  boundedness of $\{u_{n}-\overline{u} \}$ in $\mathcal V$
we obtain that 
\begin{align}\label{late1}
\limsup_{n\rightarrow +\infty}\big(\Psi(u_{n}) & -
\Psi(\overline{u}) - \langle D\Phi (u_{n}), u_{n} -
\overline{u}\rangle\big) \leq 0.
\end{align}
It now follows from (\ref{late}) and (\ref{late1}) that 
\begin{align}\label{late2}
\limsup_{n\rightarrow +\infty}\Psi(u_{n})\leq \Psi(\overline{u}).
\end{align}
On the other hand by the weak lower semi-continuity of $\Psi$ we have that 
\begin{align}\label{late3}
\liminf_{n\rightarrow +\infty}\Psi(u_{n})\geq \Psi(\overline{u}),
\end{align}
from which together with (\ref{late2}) one has that $u_n \to \overline u$ in $E_V.$ This together with  the fact that  $u_n\rightarrow \overline u$ strongly in $L^p(\mathbb R^N)$ imply that $u_n\rightarrow \overline u$ strongly in $\mathcal V.$

For each  $k\in \mathbb{N},$ considering the definition of $\Gamma_k$ in (\ref{Gam}), we define
\begin{align*}
c_{k} = \inf_{A\in \Gamma_{k}} \sup_{u \in A} \mathcal I(u).
\end{align*}
We shall now prove  that $-\infty < c_{k} < 0$ for all $k \in \mathbb{N}.$ From Lemma \ref{selfadjoint} and the  compactness of the embedding, the spectrum of
the Schr\"odinger operator $-\Delta+V$ on $L^2(\mathbb R^N)$ is discrete and consists of eigenvalues
of finite multiplicity, $0 < \mu_1<\mu_2\leq \mu_3\leq .....,$ and $\mu_k \rightarrow \infty$ as $k\rightarrow \infty$.
 To this, let us denote by $\mu_{j}$ the j-th eigenvalue of $-\Delta+V(x)$ (counted according to its multiplicity) and by $e_{j}$ a corresponding eigenfunction satisfying
 $$ \displaystyle \int_{\mathbb R^N} \nabla e_{i}.\nabla  e_{j}\,\mathrm{d}x +\int_{\mathbb R^N}V(x) e_{i} e_{j}\,\mathrm{d}x= \delta_{ij}.$$

 As in the proof of Theorem \ref{thm1}, we have that  $\mathcal I_K$ is bounded below. Thus $ c_{k} > -\infty$ for each $k \in \mathbb{N}.$ Let
\begin{align*}
A:= \Big \{u = \alpha_{1}e_{1} + ... + \alpha_{k}e_{k}  : \| u \|_{E_V}^{2} = \displaystyle \sum_{i=1}^k\alpha_{i}^{2}= \rho^{2}\Big \},
\end{align*}
for small $\rho > 0$ to be determined. Then $A \in \Gamma_{k}$ because $\gamma(A) = k$ by Proposition \ref{genus}. Since $A$ is finite dimensional, all norms are equivalent on $A$. Thus, for any $u \in A$, $\|u\|_{L^\infty(\mathbb R^N)}\leq C\|u\|_{E_V}=C\rho\leq r_2$, for sufficiently small $\rho>0$. Hence $A \subseteq K$ for suitable choice of  $\rho$.
Also there exist positive  constants $c_{1}$, $c_{2}$ such that $\| u \|_{L^{p}(\mathbb R^N)} \geq  c_{1} \| u \|_{E_V}$ and $\| u \|_{L^{q}(\mathbb R^N)} \geq  c_{2} \| u \|_{E_V}$ for all $u \in A$. Therefore,
\begin{align*}
 \mathcal I_K(u) &= \frac{1}{2}\| u \|_{E_V}^{2} - \dfrac{1}{p}\| u \|_{L^{p}(\mathbb R^N)}^{p} - \dfrac{\lambda}{q}\| u \|_{L^{q}(\mathbb R^N)}^{q} \\
 & \leq  \dfrac{1}{2} \rho^{2} - \dfrac{1}{p}c_{1}^{p} \rho^{p} - \dfrac{\lambda}{q} c_{2}^{q} \rho^{q} = \rho^{q} (\dfrac{1}{2}\rho^{2 - q} - \dfrac{1}{p}c_{1}^{p} \rho^{p - q} - \dfrac{\lambda}{q} c_{2}^{q}).
\end{align*}
Now we can choose $\rho$ small enough such that $ \mathcal I_K(u) \leq \rho^{q} (\dfrac{1}{2}\rho^{2 - q} - \dfrac{1}{p}c_{1}^{p} \rho^{p - q} - \dfrac{\lambda}{q} c_{2}^{q}) < 0$ for every  $u \in A$. It then follows that $c_{k} < 0$. Thus, by Theorem \ref{critical theorem}, $\mathcal I_K$ has  a sequence of distinct critical points $\{u_k\}_{k \in \mathbb N}$ by means of Definition \ref{cpd}. Also, by Lemma \ref{3.3}, for each critical point $u_k$ of $\mathcal I_K$ there exists $v_k \in K$ such that $-\Delta v_k+V(x)v_k=D \Phi(u_k).$ It now follows from Theorem \ref{12v2} that $\{u_k\}$ is a sequence of distinct solutions of (\ref{P}) such that $\mathcal I_K(u_k)<0$ for each $k \in \mathbb N.$ This completes the proof.
\hfill$\square$

\section{Two dimensional case}
 In order to study \eqref{P} for $N=2$, we adopt the truncation idea as follows. Define
 \[
 K(r)=\big \{u\in E_V\cap L^\infty(\mathbb R^2) :\|u\|_{L^\infty(\mathbb R^2)}\leq r\big \}
 \]
and $g: \mathbb R \to \mathbb R$ by 
\[
g(t)=
\begin{cases}
f(t),& |t|\leq r,\\
\frac{f(r)}{r}t, & |t|\geq r.
\end{cases}
\]
Thus, $g(u)=f(u)$ for $u\in K(r)$.

Therefore, our aim is to find  a solution of the following truncated problem
\begin{equation}\label{P2T}\tag{$T_{\lambda}$}
\left\{\begin{array}{lll}
&-\Delta u+V(x)u=g(u)+\lambda |u|^{q-2}u \;\;x\in \mathbb R^2,\\
&u\in H^1(\mathbb R^2),  \displaystyle \int_{\mathbb R^2}V(x)|u|^2\,\mathrm{d}x<\infty,
\end{array}\right.
\end{equation}
in $K(r)$ for some suitable choice of $r>0$. We shall apply  Theorem \ref{con2} in the following variational set-up. Let  $\mathcal J : E_V\rightarrow \mathbb{R}$ be the Euler-Lagrange functional related to \eqref{P2T}, given as
\begin{equation*}
\mathcal  J(u) = \frac{1}{2} \int_{\mathbb R^2} | \nabla u | ^{2} \,\mathrm{d}x+\frac{1}{2} \int_{\mathbb R^2}V(x)|u|^2\,\mathrm{d}x - \int_{\mathbb R^2}G(u) \,\mathrm{d}x - \frac{\lambda}{q} \int_{\mathbb R^2} | u | ^{q} \,\mathrm{d}x,
\end{equation*}
where $G(t)=\displaystyle \int_0^t g(s)ds$ is the primitive of $g(t)$. Now,
define  the function $\Upsilon : E_V\rightarrow \mathbb{R}$ by
\begin{align*}
\Upsilon (u) = \int_{\mathbb R^2}G(u) \,\mathrm{d}x+ \frac{\lambda}{q}\int_{\mathbb R^2} | u | ^{q} \,\mathrm{d}x.
\end{align*}
Note that $\Upsilon \in C^1(E_V;\mathbb R).$
Finally, let us introduce  the functional $\mathcal J_K: E_V \to (-\infty, +\infty]$ defined by
 \begin{eqnarray}\label{resfun}
 \mathcal J_K(u):= \Psi_K(u)-\Upsilon(u),
 \end{eqnarray}
 where $\Psi_K$ is defined as in \eqref{psitru} for 
 \[
 \Psi(u)= \frac{1}{2} \int_{\mathbb R^2} | \nabla u | ^{2} \,\mathrm{d}x+\frac{1}{2} \int_{\mathbb R^2}V(x)|u|^2\,\mathrm{d}x.
 \]  
In the process to verify Theorem \ref{con2}, we need similar Lemmas as in the section 2. Note that the assumption $(f2)$ implies that
 for every $\epsilon_0>0,$ there exists $\delta_0>0$ such that $|f(t)|\leq \epsilon_0 |t|^{\nu},$ whenever $|t|<\delta_0.$ In particular, for fixed $\epsilon_0=1$
 there exists $\delta_1>0$ such that 
 \begin{equation}\label{ftwo}
   |f(t)|\leq |t|^{\nu},\; \textrm{whenever}\;|t|<\delta_1. \end{equation}
We will fix this $\delta_1$ to avoid any confusion.
\begin{lem} \label{3.2T}
Assume that $1 < q < 2 $ and $\delta_1>0$ as defined in \eqref{ftwo}. Then for all $u\in K(r)$ with $0<r<\delta_1$, we have
\begin{align*}
\| D\Upsilon(u) \|_{L^\infty(\mathbb R^2)} \leq   r^{\nu} + \lambda r^{q -1}.
\end{align*}
\end{lem}

\textbf{Proof.}
By definition of $D\Upsilon (u)$ we have
\begin{align*}
\| D\Upsilon (u) \|_{L^\infty(\mathbb R^2)} &= \big \| g(u) +\lambda u | u | ^{q - 2} \big \|_{L^\infty(\mathbb R^2)} \leq \big \|g(u)\big \|_{L^\infty(\mathbb R^2)} + \lambda \big \|| u | ^{q - 1} \big \|_{L^\infty(\mathbb R^2)}.
\end{align*}
Now, using \eqref{ftwo} and  choosing $r\leq \delta_1$, we get 
\begin{align*}
\| D\Upsilon (u) \|_{L^\infty(\mathbb R^2)}&\leq \|u\|^{\nu}_{L^\infty(\mathbb R^2)}+\lambda \| u \|_{L^\infty(\mathbb R^2)}^{q -1}\leq r^{\nu}+\lambda r^{q-1}
\end{align*}
as desired.
\hfill$\square$\\
We are now in the position to state the following result addressing condition $(ii)$ in Theorem \ref{con2} by following the similar idea as in Lemma \ref{3.3} combined with Lemma \ref{3.2}.
\begin{lem} \label{3.3N2}
Let  $1 < q < 2 < p$. Choose $\Lambda_1> 0$ in such a way that  for each   $\lambda \in (0, \Lambda_1)$ there exist positive numbers $r_1, r_2$ with $r_1<r_2 < \delta_1$ such that $r \in [r_1, r_2]$ if and only if $r^{\nu}+\lambda r^{q-1} \leq V_0r.$ Then   for each $r \in  [r_{1}, r_{2}]$  and   each $\overline{u}\in K(r)$ there exists $v \in K(r)$ such that
 \[-\Delta v+V(x)v=g(\overline u) +\lambda \overline{u} | \overline{u} | ^{q - 2} .\]
\end{lem}
\textbf{Proof of Theorem \ref{thm3}.}
Let $\Lambda_1$ be as in Lemma \ref{3.3N2} and $\lambda \in (0, \Lambda_1).$ Also, let  $r_1$ and $r_2$ be as in Lemma \ref{3.3N2} and define
\[K:=\big \{u \in K(r_2); \, \, \, u(x) \geq 0 \text{ a.e. } x \in \mathbb R^2\big \}.\]
Now we continue the proof  in the following few steps.\\
{\it Step 1.}  We show that there exists  $\overline{u} \in K$ such that $\mathcal  J_K(\overline{u}) = \inf _{u \in E_V}\mathcal J_K(u)$. Then by Proposition \ref{minimum}, we conclude that  $\overline{u}$  is a critical point of $\mathcal J_K.$\\
Set $ \sigma:= \inf _{u \in E_V}\mathcal J_K(u)$. So by definition of $\Psi_K$ for every $u \notin K$, we have $\mathcal J_K(u) = +\infty$ and therefore $\sigma = \inf _{u \in K}\mathcal J_K(u)$. For each  $u \in K,$  it follows from the compact embedding  $E_V\hookrightarrow L^1(\mathbb R^2)$ as in Lemma \ref{embd} that 
\begin{align*}
 \int_{\mathbb R^2} G(u) \,\mathrm{d}x=\int_{|u|\leq r} F(u) \,\mathrm{d}x\leq C\int_{|u|\leq r}|u|^{\nu+1}\,\mathrm{d}x\leq C\|u\|^\nu_{L^\infty(\mathbb R^2)}\int_{|u|\leq r}|u|\,\mathrm{d}x\leq  C_1\|u\|_{E_V}
\end{align*}
Therefore,
\begin{align*}
\Upsilon(u)&\leq  C_1 \| u\|_{E_{V}}+ C_2\|u\|^{q}_{E_V}.
\end{align*}
Here we have used the embedding $E_{V} \hookrightarrow L^q(\mathbb R^2)$ due to Lemma \ref{embd}.  Thus, for $u \in K$ we have that 
\begin{eqnarray}\label{coerN2}
  \mathcal J_K(u):= \Psi_K(u)-\Upsilon(u) \geq \frac{1}{2}\| u\|^2_{E_{V}} - C_1 \| u\|_{E_{V}}- C_2\|u\|^{q}_{E_V} , 
\end{eqnarray}
{from which we obtain that  $\sigma> -\infty$ .}
Now, suppose that $\{u_{n}\}$ is a minimizing sequence in $K$ such that $\mathcal J_K(u_{n})\rightarrow \sigma$. It follows from (\ref{coerN2}) and the definition of the set $K$ that  the sequence $\{u_{n}\}$ is bounded in $E_{V}\cap L^\infty(\mathbb R^2)$. Now, using standard results in Sobolev spaces, after passing to a subsequence if necessary, there exists $\overline{u} \in E_V$ such that $u_n \rightharpoonup \overline{u} $ weakly in $E_{V}$. Moreover $u_n(x)\rightarrow \overline u(x)$ in $\mathbb R^2$ pointwise almost everywhere which implies $\overline u\in L^\infty(\mathbb R^2)$ with $\|\overline u\|_{L^\infty(\mathbb R^2)}\leq r_2$. As a consequence  $\overline u\in K$.  We now show that $\Upsilon (u_n) \to \Upsilon(\overline{u}).$  Indeed, using $u_n\in K$, we have that 
\[G(u_n)=\int_{0}^{u_n}g(t)dt \leq C| u_n |^{\nu+1}  .\]
Therefore, by the strong convergence $u_n \to \overline u$ in $L^\beta(\mathbb R^2)$ for all $\beta \in [1, \infty)$  as in Lemma \ref{embd} and the dominated convergence theorem we obtain that $\Upsilon(u_n) \to \Upsilon(\overline{u}).$

Therefore, $\mathcal J_K(\overline{u})\leq \liminf_{n \to \infty }\mathcal J_K(u_{n})$. Thus, $\mathcal J_K(\overline{u}) = \sigma =\inf _{u \in E_V}\mathcal J_K(u)$ and the proof of {\it Step 1} is complete.\\
{\it Step 2.}
 In this step we show that there exists $v\in K$ such that  $$-\Delta v +V(x)v=g( \overline{u} ) +\lambda \overline{u} | \overline{u} | ^{q - 2} .$$ By  Lemma \ref{3.3N2} together with the fact that $\overline{u} \in K(r_2)$ we obtain that $v\in K(r_2).$ Also, by the maximum principle and the fact  that $$-\Delta v+V(x)v=g(u) +\lambda \overline{u} | \overline{u} | ^{q - 2} \geq 0,$$
 we obtain that $v\geq 0.$ 
It now follows from Theorem \ref{12v2} together with
{\it Step 1} and
{\it Step 2} that $\overline{u}$ is a solution of the (\ref{P2T}).  To complete the proof we shall show that $\overline{u}$ is non-trivial by proving that
 $\mathcal J_{K}(\overline{u}) = \inf_{u \in K} \mathcal J_{K}(u) < 0$.
Take $0\leq e \in K$. For $t\in [0,1],$ we have that  $te \in K$ and therefore
\begin{align*}
\mathcal J_{K}(te) & = \frac{1}{2} \int_{\mathbb R^2} | \nabla te | ^{2} \,\mathrm{d}x+\frac{1}{2} \int_{\mathbb R^2} V(x)|te | ^{2} \,\mathrm{d}x - \int_{\mathbb R^2} G(te) \,\mathrm{d}x - \frac{\lambda}{q} \int_{\mathbb R^2} | te | ^{q} \,\mathrm{d}x
\\
& \leq t^{q}\left( \frac{ t^{2 - q}}{2} \int_{\mathbb R^2} | \nabla e | ^{2} \,\mathrm{d}x +\frac{t^{2-q}}{2} \int_{\mathbb R^2} V(x)|e | ^{2} \,\mathrm{d}x - \frac{\lambda}{q} \int_{\mathbb R^N} | e | ^{q} \,\mathrm{d}x\right).
\end{align*}
 Since $1 < q < 2 $,  $\mathcal J_{K}(te)$ is negative for $t$ sufficiently small. Thus,  we can conclude that $\mathcal J_{K}(\overline{u}) < 0$. Thus, $\overline{u}$ is a non-trivial and non-negative solution of (\ref{P2T}). Moreover,  it follows from the strong maximum principle that $\overline u>0$. Finally, using the fact that $\overline u\in K$ implies $\overline u$ is a positive solution of \eqref{P}.
\hfill$\square$\\

We would like to remark that by using the same argument as in the proof of Theorem \ref{thm2}, one can also prove multiplicity for the case $N=2$.
\section*{Acknowledgements}
J. M. do \'O and P. K. Mishra  are supported in part by INCTmat/MCT/Brazil, CNPq and CAPES/Brazil. A. Moameni is pleased to acknowledge the support of the National Sciences and Engineering Research Council of Canada.



\section*{References}
\begin{thebibliography}{10}

\bibitem{MR1608057}
Adimurthi,  S.~Yadava, On the number of positive solutions of some semilinear {D}irichlet
  problems in a ball,
 Differential Integral Equations 10(6) (1997) 1157--1170.

\bibitem{MR1276168}
A.~Ambrosetti, H.~Brezis, G.~Cerami, Combined effects of concave and convex nonlinearities in some
  elliptic problems,
J. Funct. Anal. 122(2) (1994) 519--543.

\bibitem{MR1301008}
T. Bartsch, M. Willem,
On an elliptic equation with concave and convex nonlinearities,
Proc. Amer. Math. Soc. 123(11) (1995) 3555--3561.

\bibitem{MR1396658}
X. Cabr\'e,  P. Majer,
Truncation of nonlinearities in some supercritical elliptic problems,
C. R. Acad. Sci. Paris S\'er. I Math. 322(12) (1996) 1157--1162.

\bibitem{MR1491844}
J.~Chabrowski,
Elliptic variational problems with indefinite nonlinearities,
 Topol. Methods Nonlinear Anal. 9(2) (1997) 221--231.

\bibitem{MR1879863}
J.~Chabrowski, J. M. do \'O,
 On semilinear elliptic equations involving concave and convex nonlinearities,
 Math. Nachr. 233/234 (2002) 55--76.

\bibitem{MR1712564}
L. Damascelli, M. Grossi, F. Pacella,
Qualitative properties of positive solutions of semilinear elliptic
  equations in symmetric domains via the maximum principle,
 Ann. Inst. H. Poincar\'e Anal. Non Lin\'eaire 16(5)  (1999) 631--652.

\bibitem{MR1971261}
D. G. De~Figueiredo, J. P. Gossez, P. Ubilla,
Local superlinearity and sublinearity for indefinite semilinear elliptic problems,
J. Funct. Anal. 199(2) (2003) 452--467.

\bibitem{MR1491612}
J.~V. Gon\c{c}alves, O.~H. Miyagaki, Multiple positive solutions for semilinear elliptic equations in $\mathbb R^N$ involving subcritical exponents,
Nonlinear Anal. 32(1) (1998) 41--51.

\bibitem{Abbas} N. Kouhestani, A. Moameni, Multiplicity results for elliptic problems with super-critical concave and convex nonlinearites,  Calc. Var. (2018) https://doi.org/10.1007/s00526-018-1333-y.

\bibitem{MR0090795}
C. Kuratowski,
Topologie. {V}ol. {I},
Monografie Matematyczne, Tom 20. Pa\'nstwowe Wydawnictwo Naukowe, Warsaw, 1958.

\bibitem{MR2185298}
Z. Liu, Z. Q. Wang, Schr\"odinger equations with concave and convex nonlinearities, Z. Angew. Math. Phys. 56(4) (2005) 609--629.

\bibitem{Mo5} A. Moameni, A variational principle for
problems with a hint of convexity, C. R. Math. Acad. Sci. Paris 355(12) (2017) 1236-1241.



\bibitem{Mo6} A. Moameni, A variational principle for
problems in Partial differential equations and Analysis, Submitted.


\bibitem{MR1721723}
T. Ouyang, J. Shi, Exact multiplicity of positive solutions for a class of semilinear
  problem. {II},
 J. Differential Equations 158(1) (1999) 94--151.
\bibitem{Rabinowitz 1}
 P. H. Rabinowitz, Variational methods of nonlinear eigenvalue problems,
 Proc. Sym. on Eigenvalues of nonlinear problems, Edizionicremonese, Rome (1974) 143-195.

\bibitem{Rabinowitz 2} 
P. H. Rabinowitz, Some aspects of critical point theory,
 MRC Tech, Rep. 2465, Madison, Wisconsin, 1983.
 \bibitem{RAD} V. R\u{a}dulescu, D. Repov\u{s}, Combined effects in nonlinear problems arising in the study of anisotropic continuous media, Nonlinear Analysis: Theory, Methods \& Applications 75(3) (2012) 1524-1530. 

\bibitem{MR837231}
A. Szulkin,
Minimax principles for lower semicontinuous functions and
  applications to nonlinear boundary value problems,
Ann. Inst. H. Poincar\'e Anal. Non Lin\'eaire 3(2) (1986) 77--109.

\bibitem{MR1983694}
M. Tang, Exact multiplicity for semilinear elliptic {D}irichlet problems
  involving concave and convex nonlinearities, Proc. Roy. Soc. Edinburgh Sect. A 133(3) (2003) 705--717.

\bibitem{MR2557956}
T. F. Wu, Multiple positive solutions for a class of concave-convex elliptic
  problems in {$\Bbb R^N$} involving sign-changing weight, J. Funct. Anal. 258(1) (2010) 99--131.

\end{thebibliography}
\end{document}